\documentclass[12pt,a4paper]{article}
\usepackage[utf8x]{inputenc}
\usepackage{ucs}
\usepackage{amsfonts, amssymb, amsmath, amsthm}
\usepackage{color}
\usepackage{graphicx}
\usepackage[lf]{Baskervaldx} 
\usepackage[bigdelims,vvarbb]{newtxmath} 
\usepackage[cal=boondoxo]{mathalfa} 

\usepackage[width=16.00cm, height=24.00cm, left=2.50cm]{geometry}

\newtheorem{theorem}{Theorem}
\newtheorem{lemma}[theorem]{Lemma}
\newtheorem{proposition}[theorem]{Proposition}

\usepackage[colorlinks=true,linkcolor=colorref,citecolor=colorcita,urlcolor=colorweb]{hyperref}
\definecolor{colorcita}{RGB}{21,86,130}
\definecolor{colorref}{RGB}{5,10,177}
\definecolor{colorweb}{RGB}{177,6,38}

\DeclareMathOperator{\dist}{dist}
\DeclareMathOperator{\re}{Re}

\allowdisplaybreaks

\title{A Montel-type theorem for Hardy spaces of holomorphic functions}

\author{Tomás Fernández Vidal\thanks{Supported by PICT 2015-2299} \and Daniel Galicer\thanks{Partially supported by CONICET-PIP 2014-2016 and PICT 2015-2299.}  \and Pablo Sevilla-Peris\thanks{Supported by MICINN and FEDER Project MTM2017-83262-C2-1-P and MECD grant PRX17/00040.}}

\date{}

\begin{document}
	
\maketitle	

\begin{abstract}
We give a version of the Montel theorem for Hardy spaces of holomorphic functions on an infinite dimensional space. 
As a by-product, we  provide a Montel-type theorem for the Hardy space of Dirichlet series.
This approach  also gives an elementary proof of Montel theorem for the classical one-variable Hardy spaces. 
\end{abstract}


\section{Introduction}

Montel's theorem  is one of the basic results in the classical function theory of one complex variable (see e.g. \cite[Theorem~2.9]{conway1978functions}). This theorem states that a family of holomorphic functions defined on an open subset of the complex numbers is normal if and only if it is locally uniformly bounded.
Therefore, every sequence of holomorphic functions on some open set $\Omega$ of the complex plane that is uniformly bounded on the compact subsets of $\Omega$ has a subsequence that converges uniformly (on the compact subsets) to some holomorphic function defined on $\Omega$. \\
We consider here the Hardy space of holomorphic functions on the disc $H_{p}(\mathbb{D})$, with $1 \leq p \leq \infty$ (see, for example, \cite[Chapter~20]{conway2012functions}). These (for each $p$) consist of those holomorphic functions $f : \mathbb{D} \to \mathbb{C}$ so that
\begin{equation}\label{norma fin var}
\Vert f \Vert_{H_{p}(\mathbb{D})} 
= \sup_{0<r<1} \bigg(\int_{\mathbb{T}}  |f(rw)|^{p} \mathrm{d}w \bigg)^{\frac{1}{p}}<\infty \,.
\end{equation}
These spaces are closely related with harmonic analysis. If $H_{p}(\mathbb{T})$ denotes the closed subspace of $L_{p}(\mathbb{T})$ (where the normalised Lebesgue measure on $\mathbb{T} = \{ z \in \mathbb{C} \colon \vert z \vert =1 \}$ is taken) consisting of those functions $f$ for which the Fourier coefficient $\hat{f}(n)$ is $0$ whenever $n<0$. A classical, well known 
result states that $H_{p}(\mathbb{T}) = H_{p}(\mathbb{D})$ as Banach spaces (see again \cite[Chapter~20]{conway2012functions}).\\
It is not difficult to establish a sort of Montel-type theorem for these Hardy spaces. If we have a bounded sequence in $ H_{p}(\mathbb{D})$, then by the Banach-Alaoglu theorem we can extract a weakly$^{*}$-convergent subsequence. Weak$^{*}$-convergence on $H_{p}(\mathbb{D})$ is equivalent to being bounded and uniformly convergent on the compact sets of $\mathbb{D}$ \cite[Chapter~20, Proposition~3.15]{conway2012functions}. This altogether gives the following result,  well known within the area, although we could not find it explicitly in the literature.

\begin{theorem} \label{una variable}
Let $1 \leq p < \infty$. If $(f_n)\subseteq H_{p}(\mathbb{D})$ is a bounded sequence, then there exist  $f\in H_{p}(\mathbb{D})$ and a subsequence $(f_{n_k})$ that converges to $f$ uniformly on the compact subsets of $\mathbb{D}$.
\end{theorem}

The same argument gives a version of this result for $H_{p}(\mathbb{D}^{n})$, the Hardy spaces on the $n$-dimensional polydisc. Our aim in this note is to obtain an analogue of this Montel-type theorem for 
holomorphic functions on infinite dimensional spaces (Theorem~\ref{verdejo}). We 
approach the proof from a different point of view. We will see how this new perspective also works to give a different  proof of Theorem~\ref{una variable}, which uses only elementary tools of the classical function theory of one complex variable, and therefore  avoiding the Banach-Alaoglu result. Finally, we will see how all this can be used to get a Montel-type theorem for Hardy spaces of Dirichlet series (Theorem~\ref{teo montel series hardy}).

\section{The result}

As we said, our aim is to provide an extension of Theorem~\ref{una variable} to an infinite-dimensional space. The first step is to find a convenient setting. First of all, let us recall that, given a complex Banach space $X$ and an open $U \subseteq X$, a function $f : U \to \mathbb{C}$ is said to be holomorphic (see e.g. \cite[Chapter~15]{defant2018Dirichlet}) if it is  Fr\'echet differentiable at every point of $U$, that is if for every $x \in U$ there exists a functional $x^{*} \in X^{*}$ so that
\[
\lim_{h \to 0} \frac{f(x+h)-f(x)-x^{*}(h)}{\Vert h \Vert} = 0 \,.
\]
Now we need a proper `infinite-dimensional version' of $\mathbb{D}$ and $\mathbb{D}^{n}$. To do this we take any Banach sequence space $X$ (i.e. a subspace of $\mathbb{C}^{\mathbb{N}}$ that contains all the canonical vectors $(e_{n})_{n}$, endowed with a complete norm such that $\Vert e_{n} \Vert =1$ for all $n$ and, if $x=(x_{n})_{n} \in \mathbb{C}^{\mathbb{N}}$ and $y=(y_{n})_{n} \in X$ are sequences so that $\vert x_{n} \vert \leq \vert y_{n} \vert$ for every $n$, then $x \in X$ and $\Vert x \Vert \leq \Vert y \Vert$) and consider the open set
\[
X \cap \mathbb{D}^{\mathbb{N}} = \{ z = (z_{n})_{n} \in X \colon \vert z_{n} \vert < 1 \text{ for all } n \} \,.
\]
For each $1 \leq p < \infty$, the Hardy space $H_{p}(X \cap \mathbb{D}^{\mathbb{N}})$ is defined as consisting of all holomorphic functions $f: X \cap \mathbb{D}^{\mathbb{N}} \to \mathbb{C}$ for which
\begin{equation}\label{norma inf var}
\Vert f \Vert_{H_{p}(X \cap \mathbb{D}^{\mathbb{N}})} =
\sup_{n \in \mathbb{N}} \sup_{0<r<1} \bigg(\int_{\mathbb{T}^n}  |f(rw_1,\cdots,rw_n,0,\cdots)|^{p} \mathrm{d}w_1 \cdots \mathrm{d} w_n\bigg)^{\frac{1}{p}}<\infty \,
\end{equation}
(here again we take on the $n$-dimensional torus $\mathbb{T}^{n}$  the normalised Lebesgue measure).\\
As in the one-dimensional case, these spaces are closely related with harmonic analysis. Let us explain briefly how. On the infinite politorus 
$\mathbb{T}^{\mathbb{N}}$ (which is a compact abelian group) we consider the product of the normalised Lebesgue measure (which is the Haar measure) and (see \cite[Chapter~8]{rudin1962fourier}) 
$H_{p}(\mathbb{T}^{\mathbb{N}})$ the class of analytic functions in $L_{p}(\mathbb{T}^{\mathbb{N}})$ (which corresponds to the classical Hardy space on $\mathbb{T}$). 
It is then well known (see \cite{badefrmase2017} or \cite[Theorem~13.2]{defant2018Dirichlet}) that $H_{p}(\ell_{2} \cap \mathbb{D}^{\mathbb{N}}) = H_{p}(\mathbb{T}^{\mathbb{N}})$
(where $\ell_{2}$ is the space of square-summable complex sequences).
These spaces $H_{p}(\ell_{2} \cap \mathbb{D}^{\mathbb{N}})$ are also closely related with the Hardy spaces of Dirichlet series introduced by Bayart in \cite{bayart2002hardy} (see Section~\ref{sD}). \\

Our aim in this note is to prove the following Montel-type result.
\begin{theorem}\label{verdejo}
Let $X$ be a Banach sequence space with $X \hookrightarrow \ell_{2}$ (with continuous inclusion) and $1 \leq p < \infty$. If
$(f_n)\subseteq H_{p}(X \cap \mathbb{D}^{\mathbb{N}})$ is a bounded sequence, then there exist  $f\in H_{p}(X \cap \mathbb{D}^{\mathbb{N}})$ and a subsequence $(f_{n_k})$ that converges to $f$ uniformly on the compact subsets of $X \cap \mathbb{D}^{\mathbb{N}}$.
\end{theorem}

If $X \hookrightarrow \ell_{2}$ then \cite[Remark~13.22]{defant2018Dirichlet} shows that $H_{p}(X \cap \mathbb{D}^{\mathbb{N}}) =  H_{p}(\ell_{2} \cap \mathbb{D}^{\mathbb{N}})$ for every 
$1 \leq p < \infty$, and every compact set in $X \cap \mathbb{D}^{\mathbb{N}}$ is also compact in $\ell_{2} \cap \mathbb{D}^{N}$. Then Theorem~\ref{verdejo} follows as a straightforward consequence 
of the following result.

\begin{proposition}\label{prop: Montel para funciones Hp}
Let $1 \leq p < \infty$. If $(f_n)\subseteq H_{p}(\ell_2\cap \mathbb{D}^{\mathbb{N}})$ is a bounded sequence, then there exist  $f\in H_{p}(\ell_2\cap \mathbb{D}^{\mathbb{N}})$ and a subsequence $(f_{n_k})$ that converges to $f$ uniformly on the compact subsets of $\ell_2\cap \mathbb{D}^{\mathbb{N}}$.
\end{proposition}

We prove this in several steps. First we use a diagonal procedure to find a subsequence that converges pointwise on some dense subset of $\ell_{2} \cap \mathbb{D}^{\mathbb{N}}$. In a second step we see that this sequence is uniformly Cauchy on the compact subsets of $\ell_{2} \cap \mathbb{D}^{\mathbb{N}}$ and, hence, converges. Finally we show that the limit function belongs to the Hardy space.
We need two results that we state now without proof. The first one is \cite[Lemma~2.16]{defant2018Dirichlet}, and provides a  locally Lipschitz condition  on compact subsets.

\begin{lemma} \label{lem: 2.16}
Let $X$ be a Banach space, $U\subseteq X$ an open set, and $K\subseteq U$ a compact set. If $f:U\rightarrow\mathbb{C}$ is holomorphic and bounded, then for every $0<s<r=\dist(X \setminus U,K)$, all $x\in K$ and $y\in\overline{B(x,s)}$
\[
|f(x)-f(y)|\leq \frac{1}{r-s}\Vert x-y\Vert \sup_{z \in U} \vert f(z) \vert.
\]
\end{lemma}


The second lemma that we need (the proof of which can be found in  \cite[Corollary~13.20 and (13.25)]{defant2018Dirichlet}), allows us to estimate $\vert f(z) \vert$ for  $z \in \ell_2\cap \mathbb{D}^{\mathbb{N}}$ in terms of  $\Vert z\Vert_2=(\sum_{j=1}^{\infty} \vert z_j \vert^2)^{1/2}$ and $\Vert z\Vert_{\infty}:=\sup_{j \in \mathbb{N}}\vert z_j \vert$.

\begin{lemma}  \label{nessun}
If $1 \leq p < \infty$, then
\begin{equation*}
|f(z)| 
\leq e^{\frac{\Vert z\Vert_2^2}{1-\Vert z\Vert_{\infty}^2}} \Vert f \Vert_{H_{p} (\ell_{2} \cap \mathbb{D}^{\mathbb{N}}) } \,.
\end{equation*}
for every $f \in H_{p} (\ell_2\cap \mathbb{D}^{\mathbb{N}})$ and $z \in \ell_2\cap \mathbb{D}^{\mathbb{N}}$.
\end{lemma}

\begin{proof}[Proof of Proposition~\ref{prop: Montel para funciones Hp}]
We begin by finding a subsequence $(f_{n_k})$ of $f_{n}$ that converges pointwise on some dense subset of $\ell_2\cap \mathbb{D}^{\mathbb{N}}$. Fix, then, some $\{x_m:\;m\in\mathbb{N}\}$ dense in $\ell_2\cap \mathbb{D}^{\mathbb{N}}$. By Lemma~\ref{nessun}, $(f_{n}(x_1))_{n}$ is bounded in $\mathbb{C}$ and we can find a subsequence $f_{n(k,1)}$ and $c_1\in\mathbb{C}$, so that 
\[
c_1=\lim_{k\to 1} f_{n(k,1)}(x_1) \, \text{ and } \, |f_{n(k,1)}(x_1)-c_1|<\frac{1}{k} \, \text{ for all } \, k\in\mathbb{N}\,.
\] 
Suppose that we have found subsequences $\big(n(k,1) \big)_{k=1}^{\infty}
, \ldots , \big(n(k,m) \big)_{k=1}^{\infty}$ in such a way that, for each $1 \leq j \leq m$,  $\big(n(k,j)\big)_{k=1}^{\infty}$ is a subsequence of $\big(n(k,j-1)\big)_{k=1}^{\infty}$ and
\begin{equation} \label{zambujo}
c_j=\lim_{k\to \infty} f_{n(k,j)}(x_{j}) \, \text{ and } \, |f_{n(k,j)}(x_j)-c_{j}|<\frac{1}{k} \, \text{ for all } k\in\mathbb{N} \,.
\end{equation}
Again, $f_{n(k,m)}(x_{m+1})$ is bounded in $\mathbb{C}$ and then there is a subsequence $f_{n(k,m+1)}$ that converges to a certain $c_{m+1}\in\mathbb{C}$, with $|f_{n(k,m+1)}(x_{m+1})-c_{m+1}|<\frac{1}{k}$ for all 
$k\in\mathbb{N}$. In this way we define $(c_{j})_{j} \subseteq \mathbb{C}$ and $(n(k,j))_{k,j} \subseteq \mathbb{N}$ in such a way that  $\big(n(k,j)\big)_{k=1}^{\infty}$ is a subsequence of 
$\big(n(k,j-1)\big)_{k=1}^{\infty}$ and \eqref{zambujo} holds for every $j$. We define now $n_{k} =n(k,k)$ and observe that for each fixed $m$, if $k \geq m$ then
\begin{equation} \label{zimmermann}
|f_{n_{k}}(x_m)-c_m|=|f_{n(k,k)}(x_m)-c_m|<\frac{1}{k} \,.
\end{equation}
Thus $\big(f_{n_k}(x_m)\big)_{k=1}^{\infty}$ converges to $c_{m}$ for every $m\in\mathbb{N}$, and we have found the subsequence that we were looking for. \\
The second step of the proof is to see that $(f_{n_{k}})_{k}$ converges to some $f \in H_{p} (\ell_{2} \cap B_{c_{0}})$ uniformly over the compact sets. To do this take some compact $K\subseteq \ell_2\cap \mathbb{D}^{\mathbb{N}}$
and set $0<r=\dist  \big(\ell_2 \setminus (\ell_2\cap \mathbb{D}^{\mathbb{N}} ),K \big)$. 
Fix now $0 \neq z \in K$  and fix $j$ so that $z_{j} \neq 0$. Define 
$y_{i} = z_{i}$  if $i \neq j$ and $y_{j} = \frac{z_j}{\vert z_{j} \vert}$; then clearly $y = (y_{i})_{i} \in \ell_2 \setminus (\ell_2\cap \mathbb{D}^{\mathbb{N}} )$ and, therefore,
\[
r \leq \Vert y-z \Vert_2 = \bigg| z_j-\frac{z_j}{|z_j|} \bigg|= \bigg| \frac{z_j|z_j|-z_j}{|z_j|} \bigg|
=\big\vert \vert z_j \vert - 1 \big\vert=1-|z_j| <1 \,.
\]
This shows that $\Vert z \Vert_{\infty} \leq 1-r$  for every $z \in K$ (since the inequality obviously holds also for $z=0$). Consider now $B=\bigcup_{z\in K} B(z,\frac{r}{2})$. Given $x \in B$ we can find $z \in K$ so that $\Vert x - z \Vert_{\infty} \leq \Vert x-z \Vert_2<\frac{r}{2}$ and, then,
\[
\Vert x \Vert_{\infty}<\Vert z\Vert_{\infty}+\frac{r}{2}\leq1-\frac{r}{2}<1 \,.
\]
With this, we can find a constant $\lambda_{B} >0$ so that $\Vert x \Vert_{2} \leq \lambda_{B}$ for every $x \in B$ and, denoting $M= \sup_{n} \Vert f_{n} \Vert_{H_{p}(\ell_{2} \cap B_{c_{0}})}$, this and Lemma~\ref{nessun} yield
\[
|f_{n_{k}}(x)| \leq
e^{\frac{\Vert x\Vert_2^2}{1-\Vert x\Vert_{\infty}^2}}\Vert f_{n_{k}}\Vert_{\mathcal{H}^{p}(\ell_{2} \cap \mathbb D^{\mathbb N})}
\leq e^{\frac{{\lambda}_{B}^{2}}{r-\frac{r^{2}}{4}}}M
\]
for every $x \in B$. This shows that each $f_{n_{k}}$ is bounded on $B$. But $B$ is open and contains $K$ so, if we take $r_{B}=\dist (\ell_2 \setminus B,K )>0$ and fix $0< \varepsilon <\frac{r_{B}}{2}$, from Lemma~\ref{lem: 2.16} we know that 
\[
|f_{n_{k}}(z)-f_{n_{k}}(y)|
\leq\frac{1}{r_{B}- \varepsilon} \Vert z-y\Vert\ \sup_{x \in B} \vert f_{n_{k}} (x) \vert
\leq\frac{1}{r_{B}- \varepsilon}\Vert z-y\Vert e^{\frac{{\lambda}_{B}^{2}}{r-\frac{r^{2}}{4}}} M
\]
for every $z\in{K}$ and all $y\in\overline{B(x, \varepsilon)}$. Since $\{B(z,\varepsilon):\;z\in{K}\}$ is a open cover of $K$, there exist $z_1,\cdots,z_N$ in $K$ such that 
\[
K\subseteq{B(z_1,\varepsilon)\cup\cdots\cup{B(z_{N},\varepsilon)}}\,
\]
and, by the denseness of $\{x_{m}:m\in\mathbb{N}\}$, for each $i = 1, \ldots ,N$ we can find $m_{i} \in \mathbb{N}$ so that $x_{m_{i}} \in B(z_{i}, \varepsilon)$. In this way, for every $z \in K$ there are 
$z_{j}$ and $x_{m_j}$, such that $z,x_{m_j}\in{\overline{B(z_j,\varepsilon)}}$ and, then,
\begin{multline*}
{}|f_{n_k}(z)-f_{n_k}(x_{m_j})| \leq|f_{n_k}(z)-f_{n_k}(z_j)|+|f_{n_k}(z_j)-f_{n_k}(x_{m_j})|\\
\leq 2M e^{\frac{{\lambda}_{B}^{2}}{r-\frac{r^{2}}{4}}} \left(\frac{\Vert z-z_j\Vert}{r_{B}-\varepsilon}+\frac{\Vert x_{m_j}-z_j\Vert}{r_{B}-\varepsilon}\right)
\leq\frac{8 e^{\frac{{\lambda}_{B}^{2}}{r-\frac{r^{2}}{4}}} M}{r_{B}}\varepsilon \,.
\end{multline*}
Now, if $k\geq l \geq \max\{m_1,\cdots,m_{N},\frac{1}{\varepsilon}\}$, using \eqref{zimmermann} we get
\begin{multline*}
|f_{n_k}(z)-f_{n_l}(z)| \\ \leq|f_{n_k}(z)-f_{n_k}(x_{m_j})|+|f_{n_k}(x_{m_j})-c_{m_j}|+|f_{n_l}(x_{m_{j}})-c_{m_j}|
+|f_{n_l}(z)-f_{n_l}(x_{m_j}) | \\
<\frac{2}{l}+\frac{8 e^{\frac{{\lambda}_{B}^{2}}{r-\frac{r^{2}}{4}}} M}{r_{B}}\varepsilon
\end{multline*}
This shows that the sequence $f_{n_k}$ is uniformly Cauchy on $K$ and, since $K$ was arbitrary the sequence is uniformly Cauchy on the compact subsets of $\ell_2\cap \mathbb{D}^{\mathbb{N}}$. Then, since the space of holomorphic functions endowed with the topology of uniform convergence on compact sets is complete (see e.g. \cite[Theorem~15.48]{defant2018Dirichlet}),  it converges to some holomorphic $f : \ell_2\cap \mathbb{D}^{\mathbb{N}} \to \mathbb{C}$.\\
To complete the proof it is only left to check that in fact $f\in H_{p}(\ell_2\cap \mathbb{D}^{\mathbb{N}})$. Let $M = \sup_{n} \Vert f_{n} \Vert_{H_{p}(\ell_{2} \cap \mathbb{D}^{\mathbb{N}})} >0$  and fix $n \in \mathbb{N}$ and $0<r<1$. Since $r \mathbb{T}^{n}\times\{0\}\subseteq\ell_2\cap \mathbb{D}^{\mathbb{N}}$ is compact,  there is $k= k(n,r) \in\mathbb{N}$ such that $|f(z)-f_{n_k}(z)|\leq\frac{1}{2}$  for every $z \in r \mathbb{T}^{n}\times\{0\}$ and, therefore
\begin{align*}
\bigg(\int_{\mathbb{T}^{n}} |f(rw_1,&\dots,rw_{n},0,\dots)|^{p} \mathrm{d}w_1\cdots\mathrm{d}w_{n}\bigg)^{\frac{1}{p}}\\
\leq& \bigg( \int_{\mathbb{T}^{n}} |f(rw_1,\dots,rw_{n},0,\dots)-f_{n_{k}}(rw_1,\dots,rw_{n},0,\dots)|^{p} \mathrm{d}w_1\cdots\mathrm{d}w_{n}\bigg)^{\frac{1}{p}}\\
&+\bigg(\int_{\mathbb{T}^{n}} |f_{n_{k}}(rw_1,\cdots,rw_{n},0,\cdots)|^{p} \mathrm{d}w_1\cdots\mathrm{d}w_{n}\bigg)^{\frac{1}{p}}\\
\leq&\bigg(\int_{\mathbb{T}^{n}} \frac{1}{2^{p}} \mathrm{d}w\bigg)^{\frac{1}{p}}
+
\Vert f_{n_{k}} \Vert_{H_{p}} 
\leq \frac{1}{2}+M.
\end{align*}
Then we have that $$\sup_{n \in \mathbb{N}} \sup_{0<r<1} \bigg(\int_{\mathbb{T}^n}  |f(rw_1,\cdots,rw_n,0,\cdots)|^{p} \mathrm{d}w_1 \cdots \mathrm{d} w_n\bigg)^{\frac{1}{p}} \leq \frac{1}{2} +M,$$
and shows that, in fact, $f \in H_{p}(\ell_2\cap \mathbb{D}^{\mathbb{N}})$.
\end{proof}

Theorem~\ref{una variable} can be deduced from Theorem~\ref{prop: Montel para funciones Hp}. Let us briefly explain how. First of all, it is pretty straightforward to see that  $H_{p}(\mathbb{D})$ can be isometrically embedded in $H_{p}(\ell_2\cap \mathbb{D}^{\mathbb{N}})$. On the other hand, 
every holomorphic function $f :\ell_2\cap \mathbb{D}^{\mathbb{N}} \to \mathbb{C}$ defines a family of coefficients in the following way (see \cite{DeMaPr09} or \cite[Chapter~15]{defant2018Dirichlet}). For each $n$ and $\alpha \in \mathbb{N}_{0}^{n}$ (where $\mathbb{N}_{0} = \mathbb{N} \cup \{0\}$) the $\alpha$-th coefficient of $f$ is defined as
\begin{equation} \label{c-alfa}
c_{\alpha}(f):=\frac{1}{(2\pi{i})^{n}} \int\limits_{|\lambda_{1}|=r}\cdots\int\limits_{|\lambda_{n}|=r} \frac{f(\lambda_1,\cdots,\lambda_N, 0,\dots)}{\lambda_1^{\alpha_1+1}\cdots\lambda_{n}^{\alpha_{n}+1}} \mathrm{d}\lambda_{n}\cdots\mathrm{d}\lambda_1\,.
\end{equation}
Then, denoting $\mathbb{N}_{0}^{(\mathbb{N})} = \bigcup_{n \in \mathbb{N}} \mathbb{N}_{0}^{n} \times \{0\}$ we have a unique family of coefficients $\big( c_{\alpha}(f) \big)_{\alpha \in \mathbb{N}_{0}^{(\mathbb{N})}}$ associated to $f$. \\
If we start with a bounded sequence $(f_{n})_{n}$ in $H_{p}(\mathbb{D})$ we may look at it as belonging to $H_{p}(\ell_2\cap \mathbb{D}^{\mathbb{N}})$, and Theorem~\ref{prop: Montel para funciones Hp} gives us a subsequence $(f_{n_{k}})_{k}$ converging 
uniformly on the compacts (of $\ell_2\cap \mathbb{D}^{\mathbb{N}}$) to some $f \in H_{p}(\ell_2\cap \mathbb{D}^{\mathbb{N}})$. But \eqref{c-alfa} and the uniform convergence on compacts show that 
$c_{\alpha}(f_{k})$ converges to $c_{\alpha}(f)$. Hence, these coefficients are $0$ for every $\alpha$ for which there is some $N \geq 2$ with $\alpha_{N} \neq 0$. Then it is also easy to see (using e.g. 
\cite[Theorem~13.2]{defant2018Dirichlet}) that in fact $f \in H_{p}(\mathbb{D})$ and, since every compact set in $\mathbb{D}$ is compact in $\ell_2\cap \mathbb{D}^{\mathbb{N}}$ the argument is completed.\\

This is, however, too long a way to prove Theorem~\ref{una variable} (go to an infinite dimensional space in order to come back to dimension 1). As a matter of fact the strategy to prove Proposition~\ref{prop: Montel para funciones Hp} 
can be adapted (and simplified) to give a direct proof of the one-dimensional result. The replacements for Lemmas~\ref{lem: 2.16} and~\ref{nessun} are direct consequences of Cauchy's Integral Formula (as in, e.g. 
\cite[Theorem~5.4]{conway1978functions}). If $f \in H_{1} (\mathbb{D})$ and $z \in \mathbb{D}$ we may take any $\vert z \vert < r <1$ to have
\[
\vert f(z) \vert =
\bigg\vert \frac{1}{2 \pi i} \int_{\vert \zeta \vert = r} \frac{f(\zeta)}{\zeta - z} \mathrm{d}\zeta \bigg\vert
\leq  \frac{r}{r - \vert z \vert} \int_{\mathbb{T}} \vert f(rw) \vert \mathrm{d}w
\leq  \frac{r}{r - \vert z \vert} \Vert f \Vert_{H_{1} (\mathbb{D})} \,.
\]Since this holds for every $r$ we have
\begin{equation} \label{andueza}
\vert f (z) \vert \leq \frac{1}{1- \vert z \vert}  \Vert f \Vert_{H_{1}(\mathbb{D})}
\end{equation}
for every $f \in H_{1}(\mathbb{D})$ and $z \in \mathbb{D}$. \\
Suppose now that $f : \mathbb{D} \to \mathbb{C}$ is holomorphic and $z_{1}, z_{2} \in s\mathbb{D}$ for some $0 < s <1$. Using again Cauchy's Integral Formula we have
\begin{equation} \label{bux}
\vert f(z_{1}) - f(z_{2}) \vert = 
\frac{1}{2 \pi} \bigg\vert \int_{\vert \zeta \vert = s} f(\zeta) \frac{z_{1} - z_{2}}{(\zeta - z_{1}) (\zeta - z_{2})} \mathrm{d} \zeta \bigg\vert
\leq \vert z_{1} - z_{2} \vert\frac{s}{(s -\vert  z_{1} \vert) (s - \vert z_{2} \vert)} \sup_{\vert \zeta \vert =s} \vert f(\zeta) \vert \,.
\end{equation}

We now present an alternative and elementary proof of Theorem~\ref{una variable}.

\begin{proof}[Alternative proof of Theorem~\ref{una variable}]
We begin by taking $\{x_{m}  \colon m \in \mathbb{N} \}$ a countable dense subset of $\mathbb{D}$. By \eqref{andueza}, for each fixed $m$, the sequence $(f_{n}(x_{m}))_{n}$ is bounded in $\mathbb{C}$  (recall that $\Vert \cdot \Vert_{H_{1}} \leq \Vert \cdot \Vert_{H_{p}}$). 
With exactly the same diagonal procedure as in the proof of Proposition~\ref{prop: Montel para funciones Hp} we can define a subsequence $(f_{n_{k}})_{k}$ and $\{ c_{m}  \colon m \in \mathbb{N} \}$ so that 
$f_{n_{k}}(x_{m})  \to c_{m}$ as $k \to \infty$ for every $m$.The key point now is to see that the sequence $(f_{n_{k}})_{k}$ is uniformly Cauchy on every compact set $K \subseteq \mathbb{D}$. It suffices to take 
$K = \overline{r \mathbb{D}}$ for $0<r<1$. Fix, then, such an $r$ and $0< \varepsilon < 1-r$. By the density of the set $\{x_{m}\}_{m}$ and compactness we may find $x_{1}, \ldots, x_{N} \in \overline{r \mathbb{D}}$ 
(belonging to the dense set, these may not be the corresponding to $m=1, \ldots, N$ but we prefer to keep the notation as neat as possible) so that
$\overline{r \mathbb{D}} \subseteq \mathbb{D} (x_{1} , \varepsilon) \cup \cdots \cup \mathbb{D} (x_{N} , \varepsilon)$
(where $\mathbb{D} (x , \varepsilon)$ denotes the open disc in $\mathbb{C}$ centred in $x$ and with radius $\varepsilon$). For each $j=1, \ldots, N$ the sequence $(f_{n_{k}}(x_{j}))_{k}$ is Cauchy, so we can find $k_{0}$ so that $\vert f_{n_{k_{1}}} (x_{j}) - f_{n_{k_{2}}} (x_{j}) \vert < \varepsilon$ for every $k_{1}, k_{2} \geq k_{0}$ and all $j=1 , \ldots , N$. Given $z \in \overline{r \mathbb{D}}$, there is some $j$ so that $\vert z - x_{j} \vert < \varepsilon$. Then, taking any $r<s<1$, \eqref{bux} and \eqref{andueza} give
\[
\vert f_{n_{k}}(z) -  f_{n_{k}}(x_{j}) \vert \leq \varepsilon \frac{s}{(s-r)^{2}} \frac{1}{1-r} \Vert f_{n_{k}} \Vert_{H_{p} (\mathbb{D})} \,.
\]
We may choose $s = r +\frac{1-r}{2}$ and, then $\frac{s}{(s-r)^{2}} \frac{1}{1-r} = \frac{2(r+1)}{(1-r)^{3}} = K_{r}$. Denoting $M =  \sup_{k} \Vert f_{n_{k}} \Vert_{H_{p} (\mathbb{D})}$ we have $\vert f_{n_{k_{1}}}(z) -  f_{n_{k_{2}}}(z) \vert \leq \varepsilon ( 1 + 2K_{r}M)$ for every $z \in \overline{r \mathbb{D}}$ and $k_{1}, k_{2} \geq k_{0}$. This shows that the sequence $(f_{n_{k}})_{k}$ is uniformly Cauchy on every compact subset of $\mathbb{D}$. Since the space of holomorphic functions on
$\mathbb{D}$ (with the topology of uniform convergence on compact sets) is complete \cite[Corollary~2.3]{conway1978functions}, it converges to some holomorphic function $f : \mathbb{D} \to \mathbb{C}$. Exactly the same argument as in the proof of Proposition~\ref{prop: Montel para funciones Hp} shows that $f$ belongs to $H_{p}(\mathbb{D})$.
\end{proof}

\section{An application to Dirichlet series} \label{sD}

Dirichlet series are formal series of the form $\sum a_{n} n^{-s}$, where $a_{n} \in \mathbb{C}$ and $s$ is a complex variable. It is well known that Dirichlet series converge on half-plains and there define holomorphic functions. The space $\mathcal{H}_{\infty}$ consists of all Dirichlet series that converge on 
$[\re s >0]$ and there define a bounded holomorphic function, which the norm defined as the supremum on
$[\re s >0]$ is a Banach space. The reader is referred to \cite{defant2018Dirichlet,queffelec2013diophantine} for a complete account on this theory. \\
Given a sequence of Dirichlet series in $\mathcal{H}_{\infty}$ one can look at them as functions on the right half-plane and try to  apply Montel's theorem. This would have two problems: one would get a subsequence that 
converges uniformly on the compacts (and not on half-planes, which are the natural setting for Dirichlet series), and, moreover, it would converge to some holomorphic function on $[\re s >0]$ which might or might not 
be represented by a Dirichlet series. So one could say that the classical Montel's theorem is useless within the context of Dirichlet series. Bayart overcame this problem in  \cite[Lemma~18]{bayart2002hardy}, proving a Montel-type theorem for Dirichlet series: \textit{let $\{ \sum a_{n}^{(N)} n^{-s}\}_{N}$ be a bounded sequence in $\mathcal{H}_{\infty}$; then it has a subsequence that converges to a Dirichlet series $\sum a_{n} n^{-s}$ in $\mathcal{H}_{\infty}$ uniformly on $[\re s > \sigma]$ for every $\sigma >0$.} This has several interesting applications within the functional-analytic theory of Dirichlet series. \\
Let us look at this result from a slightly different point of view. 
If $\sum a_{n} n^{-s}$ is a Dirichlet series in $\mathcal{H}_{\infty}$ and $D(s)$ is the holomorphic function that it defines on $[\re s >0]$, then for $\varepsilon > 0$ we have
\[
\sup_{\re s > \varepsilon} \Big\vert \sum_{n=1}^{\infty} a_{n} \frac{1}{n^{s}} \Big\vert 
= \sup_{\re s > \varepsilon} \vert D(s) \vert
= \sup_{\re s > 0} \vert D(s+\varepsilon) \vert
= \sup_{\re s > 0} \big\vert \sum_{n=1}^{\infty} a_{n} \frac{1}{n^{s + \varepsilon}} \Big\vert 
= \sup_{\re s > 0} \big\vert \sum_{n=1}^{\infty} \frac{a_{n}}{n^{\varepsilon}} \frac{1}{n^{s}} \Big\vert \,.
\] 
So Bayart's Montel theorem for Dirichlet series tells us that, if $D_{N}$ is the sequence of functions associated to the Dirichlet series, then there is a subsequence and a Dirichlet series with limit function $D$ so that the translated Dirichlet series $D_{N_{k}}(s+\varepsilon)$ converges to $D(s+\varepsilon)$  uniformly on $[\re s >0]$ for every $\varepsilon >0$. Or, to put it in other terms: \textit{let $\{ \sum a_{n}^{(N)} n^{-s} \}_{N}$ be a bounded sequence in 
$\mathcal{H}_{\infty}$; then there is a subsequence $\{ \sum a_{n}^{(N_{k})} n^{-s} \}_{k}$ and a Dirichlet series $\sum a_{n} n^{s}$ in $\mathcal{H}_{\infty}$ so that $\{ \sum \frac{a_{n}^{(N_{k})}}{n^{\varepsilon}} n^{-s} \}_{k}$ converges  in $\mathcal{H}_{\infty}$  to 
$\sum \frac{a_{n}}{n^{\varepsilon}} n^{-s}$ for every $\varepsilon >0$.} \\

Hardy spaces of Dirichlet series (denoted $\mathcal{H}_{p}$, for $1 \leq p < \infty$) were introduced by Bayart in \cite{bayart2002hardy} in the following way: given $1 \leq p < \infty$, the expresion
\[
\Big\Vert \sum_{n=1}^{N} a_{n} n^{-s} \Big\Vert_{p}
= \lim_{R \to \infty} \bigg( \frac{1}{2R} \int_{-R}^{R} \Big\vert \sum_{n=1}^{N} a_{n} n^{-it} \Big\vert^{p}  dt  \bigg)^{\frac{1}{p}} \,.
\]
defines a norm on the space of Dirichlet polynomials (i.e. finite Dirichlet series). Then the space $\mathcal{H}^{p}$ is defined as the completion of the Dirichlet polynomials under this norm.
These spaces are also closely related with the Hardy spaces of holomorphic functions that we considered above. 
A Dirichlet series $\sum a_{n} n^{-s}$ belongs to $\mathcal{H}_{p}$ if and only if there exists $f \in H_{p} (\ell_{2} \cap \mathbb{D}^{\mathbb{N}})$ so that $a_{n} = c_{\alpha} (f)$ (recall \eqref{c-alfa}) whenever $n=  \mathfrak{p}^{\alpha}$ (where $\mathfrak{p} = (\mathfrak{p}_{k})_{k}$ is the sequence of prime numbers), and both have the same norm (see \cite[Theorem~3.9]{badefrmase2017} or \cite[Corollary~13.3]{defant2018Dirichlet}). In other words, 
\begin{equation} \label{Hps}
\mathcal{H}_{p} = H_{p}(\ell_2\cap \mathbb{D}^{\mathbb{N}})
\end{equation}
as Banach spaces.\\

Our aim now is to show the following version of Bayart's Montel-type theorem for Hardy spaces.

\begin{theorem}\label{teo montel series hardy}
Let $\{ \sum{a_n^{(N)}n^{-s}} \}_{N}$ be a bounded sequence in $\mathcal{H}^{p}$. Then there exist $\sum{a_{n}n^{-s}}\in\mathcal{H}^{p}$, and a subsequence $\{ \sum a_{n}^{(N_{k})} n^{-s} \}_{k}$ such that 
$\{ \sum \frac{a_{n}^{(N_{k})}}{n^{\varepsilon}} n^{-s} \}_{k}$ converges to $\sum \frac{a_{n}}{n^{\varepsilon}} n^{-s}$ in $\mathcal{H}^{p}$
for every $\varepsilon >0$.
\end{theorem}

Note that $\varepsilon=0$ cannot be taken in Theorem \ref{teo montel series hardy}. For example, given $1\leq p < \infty$ the sequence of monomials $(n^{-s})_{n}$ is bounded ($ \Vert n^{-s} \Vert_{\mathcal{H}_{p}}=1$) but it does not have a convergent subsequence. \\

Recently in \cite[Theorem~4.20]{defantschoolmann_2019}, using techniques of harmonic analysis on compact abelian groups, a Montel-type theorem has been independently obtained in the more general setting of general Dirichlet series. 
Nevertheless there is no statement written explicitly for the classical Hardy spaces of Dirichlet series.
Also, Bayart \cite{bayart2019personal} has drawn our attention to a direct proof of Theorem~\ref{teo montel series hardy} based on the $p=2$ case (somewhat easier to handle) and an argument through translation of Dirichlet series. 
We go here a different way, showing how it can be obtained also from Theorem~\ref{prop: Montel para funciones Hp}.\\

Let us make first a short comment that we will use within the proof. First of all, let us recall that there is a constant $C>0$ so that, for every $1 \leq p <\infty$, all $x \geq 2$ and $\sum a_{n} n^{-s} \in \mathcal{H}^{p}$,
\begin{equation} \label{2/4}
\bigg\Vert \sum_{n \leq x} a_{n} n^{-s} \bigg\Vert_{\mathcal{H}^{p}}
\leq C \log x \bigg\Vert \sum a_{n} n^{-s} \bigg\Vert_{\mathcal{H}^{p}}
\end{equation}
(see \cite[Theorem~12.5]{defant2018Dirichlet}). If we translate the series a little bit we can say more.\\
If a Dirichlet series $\sum a_{n}n^{-s}$ belongs to $\mathcal{H}^{p}$ for some $1 \leq p < \infty$, then the translated series $\sum \frac{a_{n}}{n^{\varepsilon}}n^{-s}$ belongs to $\mathcal{H}^{q}$ for all $\varepsilon>0$ and every $p<q<\infty$ (see \cite[Section~3]{bayart2002hardy} or \cite[Theorem~12.9]{defant2018Dirichlet}). 
This, combined with the fact that 	
the monomials  $\{n^{-s} \}_{n}$ form a Schauder basis of $\mathcal{H}^{q}$ (see \cite{alemanolsensaksman2014}) for $1 < q < \infty$ and the monotonicity of the $\mathcal{H}_{p}$-norms immediately imply
\begin{equation} \label{mar}
\lim_{l \to \infty} \bigg\Vert \sum \frac{a_n}{n^{\varepsilon}} n^{-s} - \sum_{n=1}^{l} \frac{a_n}{n^{\varepsilon}}n^{-s} \bigg\Vert_{\mathcal{H}^{p}} =0 
\end{equation}
for every $\varepsilon >0$. 

\begin{proof}[Proof of Theorem~\ref{teo montel series hardy}]
By \eqref{Hps}, we have a bounded sequence $(f_{N})_{N}$ in $H_{p}(\ell_2\cap \mathbb{D}^{\mathbb{N}})$ and, by Theorem~\ref{prop: Montel para funciones Hp}, a subsequence $(f_{N_{k}})_{k}$ that
converges uniformly on the compact sets to some $f \in H_{p}(\ell_2\cap \mathbb{D}^{\mathbb{N}})$. Moreover, if $\sum a_{n} n^{-s} \in \mathcal{H}^{p}$ is the Dirichlet series associated to $f$ by \eqref{Hps}, then \eqref{c-alfa} (and the uniform convergence on compact sets) yields $a_n^{(N_k)}=c_{\alpha}(f_{N_{k}})\rightarrow{c_{\alpha}(f)}=a_n$ as $k\rightarrow\infty$. \\
Fix $\varepsilon >0$ and note that, from \eqref{mar} we have that the partial sums of $\sum \frac{a_{n}}{n^{\varepsilon}} n^{-s}$ and $\sum \frac{a_{n}^{(N_{k})}}{n^{\varepsilon}} n^{-s}$ converge to the corresponding series. 
Our first aim is to show that we can control
the convergence of all these partial sums uniformly in some sense. To be more precise, our goal now is to show that for every $\eta >0$ there is some $l_{0} >0$ such that 
\begin{equation}\label{vanaert}
\Big\Vert \sum \frac{a_{n}}{n^{\varepsilon}} n^{-s} - \sum_{n=1}^{l} \frac{a_n}{n^{\varepsilon}}n^{-s} \Big\Vert_{\mathcal{H}^{p}}<\eta \quad \text{ and } \quad
\Big\Vert \sum \frac{a_{n}^{(N_{k})}}{n^{\varepsilon}} n^{-s} - \sum_{n=1}^{l} \frac{a_n^{(N_{k})}}{n^{\varepsilon}}n^{-s} \Big\Vert_{\mathcal{H}^{p}}<\eta
\end{equation}
for every $l \geq l_{0}$.
To begin with we fix $l \geq 2$ and  \eqref{mar} gives
\[
\Big\Vert \sum \frac{a_{n}}{n^{\varepsilon}} n^{-s} - \sum_{n=1}^{l} \frac{a_n}{n^{\varepsilon}} n^{-s} \Big\Vert_{\mathcal{H}^{p}}
=
\lim_{j \rightarrow \infty} \Big\Vert \sum_{n=l+1}^{j} \frac{a_n}{n^{\varepsilon}} n^{-s} \Big\Vert_{\mathcal{H}^{p}} \,.
\]
We denote $M= \max \big\{ \big\Vert \sum{a_n n^{-s}} \big\Vert_{\mathcal{H}^{p}}, \sup_{N} \big\Vert \sum{a_n^{(N)}n^{-s}} \big\Vert_{\mathcal{H}^{p}} \big\}$ and, for each $j \geq l+1$ Abel summation and \eqref{2/4} give
\begin{align*}
\bigg\Vert\bigg(  \sum_{n=l+1}^{j} a_n n^{-s} \bigg) & \frac{1}{j^{\varepsilon}}
+ \sum_{n=l+1}^{j-1} \bigg( \sum_{m=l+1}^{n} a_m m^{-s} \bigg) \bigg( \frac{1}{n^{\varepsilon}} - \frac{1}{(n+1)^{\varepsilon}} \bigg) \bigg\Vert_{\mathcal{H}^{p}}\\
\leq&  \bigg\Vert \sum_{n=l+1}^{j} a_n n^{-s} \bigg\Vert_{\mathcal{H}^{p}} \frac{1}{j^{\varepsilon}}
+ \sum_{n=l+1}^{j-1} \bigg\Vert \sum_{m=l+1}^{n} a_m m^{-s} \bigg\Vert_{\mathcal{H}^{p}} 
\bigg(\frac{1}{n^\varepsilon} - \frac{1}{(n+1)^{\varepsilon}} \bigg)\\
\leq & \bigg( \bigg\Vert \sum_{n=1}^{j} a_{n} n^{-s} \bigg\Vert_{\mathcal{H}^{p}} + \Vert \sum_{n=1}^{l} a_{n} n^{-s} \bigg\Vert_{\mathcal{H}^{p}} \bigg) \frac{1}{j^{\varepsilon}} \\
& + \sum_{n=l+1}^{j-1} \bigg( \bigg\Vert \sum_{m=1}^{n} a_m m^{-s} \bigg\Vert_{\mathcal{H}^{p}} + \bigg\Vert \sum_{m=1}^{l} a_m m^{-s} \bigg\Vert_{\mathcal{H}^{p}} \bigg)
\bigg(\frac{1}{n^\varepsilon} - \frac{1}{(n+1)^{\varepsilon}} \bigg) \\
 \leq & \big( C \log(j) M + C \log(l) M \big) \frac{1}{j^{\varepsilon}} \\
& + \sum_{n=l+1}^{j-1} \big( C \log(n) M+ C \log(l) M \big)
\bigg(\frac{1}{n^\varepsilon} - \frac{1}{(n+1)^{\varepsilon}} \bigg) \\
\leq & 2 C \log(j) M \frac{1}{j^{\varepsilon}}
+ \sum_{n=l+1}^{j-1} 2 C \log(n) M \frac{\varepsilon}{n^{\varepsilon+1}}\\
\leq & 2 C \log(j) M \frac{1}{j^{\varepsilon}} + \sum_{n=l+1}^{j-1} 2 C \log(n) M \frac{\varepsilon}{n^{\varepsilon+1}}\, .
\end{align*}
Hence
\[
\Big\Vert \sum \frac{a_{n}}{n^{\varepsilon}} n^{-s}  - \sum_{n=1}^{l} \frac{a_n}{n^{\varepsilon}} n^{-s} \Big\Vert_{\mathcal{H}^{p}}
\leq \varepsilon C M \sum_{n=l+1}^{\infty} \frac{\log(n)}{n^{\varepsilon+1}} \,.
\]
For each fixed $k$, exactly the same computations give the same inequality for $\sum a_{n}^{(N_{k})} n^{-s}$. Since the term on the right-hand-side tends to $0$ as $l \to \infty$ we can find $l_{0}$ satisfying \eqref{vanaert}.\\
Set $L=\max_{1 \leq n \leq l_{0}} \{ \Vert n^{-s} \Vert_{\mathcal{H}^{p}}\}$ and pick $k_{0} \in \mathbb{N}$, such that if $k \geq k_{0}$ then $| a_n^{(N_{k})} - a_n | < \frac{\eta}{l_{0}L}$ for all $1 \leq n \leq l_{0}$.
With all this we finally have, for $k \geq k_{0}$
\begin{align*}
\bigg\Vert \sum \frac{a_{n}^{(N_{k})}}{n^{\varepsilon}} n^{-s} - \sum \frac{a_{n}}{n^{\varepsilon}} n^{-s}  \bigg\Vert_{\mathcal{H}^{p}}
\leq & \bigg\Vert \sum \frac{a_{n}^{(N_{k})}}{n^{\varepsilon}} n^{-s} - \sum_{n=1}^{l_{0}} \frac{a_n^{(N_{k})}}{n^{\varepsilon}}n^{-s} \bigg\Vert_{\mathcal{H}^{p}} \\
& + \bigg\Vert \sum \frac{a_{n}}{n^{\varepsilon}} n^{-s}  - \sum_{n=1}^{l_{0}} \frac{a_n}{n^{\varepsilon}}n^{-s} \bigg\Vert_{\mathcal{H}^{p}} + \bigg\Vert \sum_{n=1}^{l_{0}} \frac{a_n^{(N_{k})}-a_n}{n^{\varepsilon}}n^{-s} \bigg\Vert_{\mathcal{H}^{p}}\\
\leq & 2 \eta + \sum_{n=1}^{l_{0}} \frac{|a_n^{(N_{k})}-a_n |}{n^{\varepsilon}} \big\Vert  n^{-s} \big\Vert_{\mathcal{H}^{p}} \\
\leq & 2 \eta + \sum_{n=1}^{l_{0}} |a_n^{(N_{k})}-a_n | \, \big\Vert n^{-s} \big\Vert_{\mathcal{H}^{p}}
<3\eta \,. \qedhere
\end{align*}
\end{proof}

\section*{Acknowledgements}
The authors thank Daniel Suarez for the clarifying observations regarding what is known in the theory of one-complex variable.
%

\noindent 
T.~Fern\'andez Vidal, D.~Galicer\\
Departamento de Matem\'{a}tica,
Facultad de Cs. Exactas y Naturales, Universidad de Buenos Aires and IMAS-CONICET. Ciudad Universitaria, Pabell\'on I (C1428EGA) C.A.B.A., Argentina, tfernandezvidal@yahoo.com.ar, dgalicer@dm.uba.ar\\ 

\noindent P.~Sevilla-Peris\\	
Insitut Universitari de Matem\`atica Pura i Aplicada. Universitat Polit\`ecnica de Val\`encia. Cmno Vera s/n 46022, Spain, psevilla@mat.upv.es

\end{document}